\documentclass[11pt]{article}

\title{\LARGE \bf
On the structural stability of random systems
}

\author{M. A. Belabbas$^{1}$ and A. Kirkoryan$^{2}$}

\usepackage{amsmath,enumerate,amsthm}
\usepackage{amsfonts}
\usepackage{ amssymb }
\usepackage{mathrsfs}
\usepackage{thmtools}
\usepackage{thm-restate}
\usepackage{todonotes}

\usepackage{tikz}
\usepackage{pgf}
	
\usepackage{pgfplots}
\pgfplotsset{compat=newest}

\usepackage{mathtools}
\usepackage{comment}
\usepackage{diagbox}
\usepackage{cases}
\usepackage{microtype}
\mathtoolsset{showonlyrefs=true}

\newtheorem{Definition}{Definition}[section]

\newtheorem{Proposition}{Proposition}[section]
\newtheorem{Theorem}{Theorem}[section]
\newtheorem{Lemma}{Lemma}[section]
\newtheorem{Corollary}{Corollary}[section]
\newtheorem{Remark}{Remark}[section]

\newcommand{\R}{\mathbb{R}}

\newcommand{\Span}{\operatorname{span}}

\newcommand{\bS}{\mathbf{S}}
\newcommand{\bH}{\mathbf{H}}
\newcommand{\bL}{\mathbf{L}}
\newcommand{\G}{\mathcal{G}}
\newcommand{\bP}{\mathbb{P}}
\newcommand{\cN}{\mathcal{N}}
\newcommand{\cT}{\mathcal{T}}
\begin{document}

\date{}
\maketitle
\thispagestyle{empty}
% \pagestyle{empty}

%%%%%%%%%%%%%%%%%%%%%%%%%%%%%%%%%%%%%%%%%%%%%%%%%%%%%%%%%%%%%%%%%%%%%%%%%%%%%%%%
\begin{abstract}
Understanding which system structure can sustain stable dynamics is a fundamental step in the design and analysis of large scale dynamical systems. Towards this goal, we investigate here the structural stability of systems with a random structure. As is usually done, we describe the system's structure through a graph describing interactions between  parts of the system, and we call a graph stable if it describe a structurally stable system, i.e., a system which can sustain stable dynamics. We then consider two Erd\"os-R\'enyi random graph models, and we obtain for each the probability that a graph sampled from these models is structurally stable.
\end{abstract}

\section{Introduction}

Understanding how global system properties, such as stability or controllability, depend on the underlying system's structure is a key step of the design and analysis of large scale networked systems. Indeed, a structural analysis, because it relies on a coarse description of the system,  enables one to focus on the basic characteristics  that make a property hold and to relegate the finer details to a subsequent study~\cite{egerstedtbook2010, shamma_coopcontrol_book}.

We focus here on the structural stability of linear systems. In general,  a  system structure is a set of systems. As is usually done in the case of linear systems, we consider here structures in which entries of the system matrix are either fixed to zero, or are arbitrary real~\cite{lin74}.  We say that a {\it system  structure} has property X (e.g., $X$ is stability of controllability) if there exists an open set of systems with this structure that having property X.  

In the study of structural properties of systems, the case of structural controllability is the most developed one. Necessary and sufficient conditions for a system to be structurally controllable were derived in~\cite{lin74} and, since then, the work has been extended in several directions. Among others, in~\cite{corfmat1976decentralized}, the authors considered more general structures than the ones in~\cite{lin74}, allowing entries of the system matrix to vary co-dependently. Extensions in similar directions were derived later in~\cite{willems1986structural}. In all of the above studies,  the structure of the system is predetermined. In~\cite{o2016conjecture}, the authors consider {\it random} structures with a  different focus:  a unique system   (with system matrix given by the graph Laplacian) is assigned to a structure. The authors then focus on proving an old conjecture of Godsill on graph controllability. In~\cite{shahrivar2017spectral}, a similar model involving the graph Laplacian is studied, but the focus is on {\it connectivity properties} of the underlying graph, which the authors then relate to eigenvalues of the graph Laplacian and robustness properties of the corresponding system.

In this paper, we consider the structural stability of systems with random structure: given a parametric model from which a structure is sampled, we derive the probability that the structure is stable, i.e., the probability that it contains a stable system. Our understanding of  structural stability of linear systems is as of now far less complete than the one of structural controllability. Indeed, whereas simple necessary and sufficient conditions for structural controllability of a system are known~\cite{lin74}, only a set of necessary and a set of sufficient conditions for structural stability have been found so far for general systems~\cite{belabbas2013sparse}. When restricting ourselves to the so-called reciprocal or symmetric structures (defined precisely below), our understanding is far more complete and a set of necessary and sufficient conditions for structural stability has been exhibited in~\cite{kirkoryan2014decentralized}. 

Based on these results, we consider in this paper two different random models for the system structure.  In the first model, entries of the system matrix are free with a given probability $p$, and in the second model, one places $N$ free entries in the matrix uniformly at random. These models, when the system structure is described by a graph, correspond to the well-known Erd\"os-R\'enyi random graph models. We will in particular show the existence of  thresholds for $p$ and $N$ after which most system structures sampled from these model are stable. For example, we will see that placing only about $N=n \log(n)/2$ free entries uniformly at random in a matrix of size $n \times n$ yields with high-probability a stable structure.

\paragraph{Notation}
We briefly review the asymptotic notation used in the paper. Let $n$ be a positive integer:
\begin{enumerate}
    \item $f(n) \in o(g(n))$ means that for {\it any} constant $c$, there exists $n_0$ so that $f(n)< c g(n)$ when $n > n_0$.
    \item $f(n) \in \omega(g(n))$ means that for {\it any} constant $c$, there exists $n_0$ so that $f(n)> c g(n)$ when $n > n_0$.
\end{enumerate}
In particular $f(n) \in o(1)$ means that $f(n)$ goes to zero and $f(n) \in \omega(1)$ means $f(n)$ goes to $\infty$.

\section{When is a system structure stable?}\label{sec:prelim}

Let $n \geq 0$ be a positive integer and denote by $A_{ij}$ the canonical basis vectors in $\R^{n \times n}$ (i.e. $A_{ij} \in \R^{n \times n}$ has zero entries everywhere except for the $ij$-th entry, which is one). Let $E \subseteq \cN^2$ be a set of pairs of elements in $\{1,\ldots,n\}$. The {\bf zero-pattern} $Z_E$ is the vector subspace of $\R^{n \times n}$ given by the linear span of $A_{ij}$, for  $(i,j) \in E$:
$$Z_E := \Span\{ A_{ij} \mid (i,j) \in E\}.$$ 
A zero-pattern describes a {\it system structure}.
Zero-patterns also admit the following two convenient representations: 
\begin{itemize}
\item 	as a matrix with $0/\ast$ entries, where the zero entries are fixed and the $\ast$ entries are arbitrary real numbers,
\item  and as a graph on $n$ nodes $\{1,2,\ldots,n\}$ with edge set $E$.
\end{itemize}
For example, with $n=3$ and $E=\{(1,2),(2,1),(3,2),(1,3),(2,2)\}$, the zero-pattern $Z_E$ can be represented by the $0/\ast$ matrix or the graph depicted in Fig.~\ref{fig:zpg1}.

\begin{figure}
    \centering
\begin{minipage}{0.45\textwidth}
\centering
\large
\begin{align*}
\left[\begin{matrix}
0 & \ast & \ast \\
\ast & \ast & 0  \\
0 & \ast & 0 
\end{matrix}\right]
\end{align*}
\vspace{0.1cm}
\end{minipage}
$\leftrightsquigarrow$
\begin{minipage}{0.4\textwidth}
\centering
\includegraphics{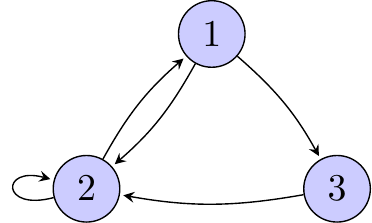}
%\begin{tikzpicture}[scale = .27, ->,>=stealth,shorten >=1pt,auto,node distance=1.8cm,
%  thin,main node/.style={circle,fill=blue!20,draw}]
%
%  \node[main node] (1) {1};
%  \node[main node] (2) [below left of=1, yshift=-.3cm] {2};
%  \node[main node] (3) [below  right of=1,yshift=-.3cm] {3};
%
%  \path[every node/.style={font=\sffamily\small}]
%    (1)      edge [bend left=10]  (2)
%      edge [bend left=10, ]  (3)
%
%    (2) edge [bend left=10](1)
%     edge [loop left] (2)   
%
%    (3)     edge [bend left=10](2)
%        ;
%\end{tikzpicture}
\end{minipage}
    \caption{There is a bijection between zero-patterns and digraphs.}\label{fig:zpg1}
\end{figure}

We are interested in this work in {\it stable} zero patterns: 
\begin{Definition}[Stable zero-patterns/graphs]
A zero-pattern is called  {\bf stable} or Hurwitz if it contains a stable matrix, i.e. a matrix whose  eigenvalues have with negative real parts. A graph corresponding to a stable zero-pattern is called stable.
\end{Definition}
The structural stability of systems was studied  in~\cite{belabbas2013sparse}, where necessary conditions and sufficient conditions for stability were exhibited. These conditions, which we recall below, are most naturally expressed in terms of the graph associated to a zero-pattern. To this end, we introduce a few graph theoretic definitions.

Let $G=(V,E)$ be a directed graph (digraph) on $n$ nodes. By convention, we set $V=\{1,\ldots,n\}$, and we denote edges by their origin and target node pairs: $(i,j)$ is the edge from node $i$ to node $j$.  We call a {\bf loop} or \textbf{self-loop} an edge of the type $(i,i)$, for $i \in V$.  A graph is called {\bf simple} if it is both undirected and without loops. We denote by $\Omega_n$ the {\bf set of all undirected graphs} on $n$ nodes, possibly with loops, and by $\Omega^s_n$ the set of all {\it simple} graphs on $n$ nodes. A subset of $I \subseteq V$ is called an {\bf independent set} if {\it none} of its vertices are connected with an edge:  i.e. $u,v \in I \Rightarrow (u,v) \notin E$. For $I \subseteq V$, we denote by $N(I) \subseteq V$ the {\bf neighbor set} of $I$, i.e. $v \in N(I)$ if there exists $u \in I$ such that $(u,v) \in E$. We illustrate these definitions in Fig.~\ref{fig:illindneigh}.

\begin{figure}\label{fig-nested}
    \centering
    \includegraphics{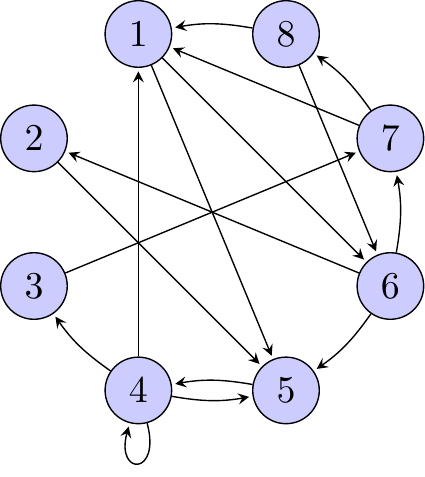}
%\begin{tikzpicture}[scale = .25,-> ,>=stealth,shorten >=1pt,auto,node distance=1.5cm,
%  thin,main node/.style={circle,fill=blue!20,draw}]
%
%  \node[main node] (1) {1};
%  \node[main node] (8) [right of=1] {8};
%  \node[main node] (7) [below  right of=8] {7};
%  \node[main node] (2) [below left of =1]{2};
%  \node[main node] (3) [below of =2]{3};
%  \node[main node] (4) [below right of =3]{4};
%  \node[main node] (5) [right of =4]{5};
%  \node[main node](6) [above right of =5]{6};
%  \path[every node/.style={font=\sffamily\small}]
%    (1) edge (5)
%        edge (6)
%
%    (2) edge (5)
%      
%
%    (3)      edge(7)
%
%      (4) edge [bend right =10](5)
%      edge [loop below] (4)
%      edge (1)
%      edge [bend left = 10] (3)
%      (5) edge[bend right = 10](4)
%      (6) edge[bend right =10](7)
%        edge [bend left = 10] (5)
%        edge (2)
%      (7) edge [bend right = 10] (8)
%        edge (1)
%      (8) edge (6)  
%      edge [bend right = 10] (1);
%
%\end{tikzpicture}
    \caption{The set \{1, 2, 3\} is an independent set in $G$ and its neighbor set is \{5, 6, 7\}. The sets \{4\}, \{4, 5\}, \{4, 5, 1\}, \{4, 5, 1, 6\}, \{4, 5, 1, 6, 7\}, \{4, 5, 1, 6, 7, 3\}, \{4, 5, 1, 6, 7, 3, 2\}, \{4, 5, 1, 6, 7, 3, 2, 8\} are nested $k$-decompositions. The last decomposition, $(1, 6, 2, 5, 4, 3, 7, 8)$ is also called Hamiltonian. According to Theorem \ref{th:stabcond}, the graph above is stable.}
    \label{fig:illindneigh}
\end{figure}

A {\bf cycle} $C$ in the graph is a sequence, without repetition, of nodes linked by edges. For example, $C=\{i_1,i_2,i_3\}$ is a cycle in $G$ if and only if $(i_1,i_2),(i_2,i_3),(i_3,i_1) \in E$ and  $i_1 \neq i_2 \neq i_3 \neq i_1$. We say that a cycle $C$ covers the set of nodes $S \subseteq V$ if $S \subseteq C$. We say that two cycles are \textbf{disjoint} if their node sets have an empty intersection. 

The following definition is central to the characterization of graph stability~\cite{belabbas2013sparse}: a decomposition in $G$ is a {\it disjoint union} of cyles of $G$, a {\bf k-decomposition} in $G$ is a decomposition that  {\it exactly} $k$ nodes, i.e. so that the union of its constituent cycles is of cardinality $k$. We also refer to $n$-decompositions as {\bf Hamiltonian decompositions}, where  we recall that $n$ is the cardinality of the node set of $G$. For $k_1 < k_2$, let $D_1$ and $D_2$ be $k_1$- and $k_2$-decompositions in $G$.  We say that $D_1$ is {\bf nested} in $D_2$, and write $D_1 \subset D_2$,  if the  node set of $D_2$ (strictly) contains the node set of $D_1$ (note that the edges used by the cycles of the decompositions play no role in the definition of nestedness). We illustrate these definitions in Fig.~\ref{fig-nested}. We can now summarize the main results of structural stability:
\begin{Theorem}[Taken from~\cite{belabbas2013sparse}]

\label{th:stabcond}
A zero-pattern $Z_E \in \mathbb{R}^{n\times n}$ with a corresponding directed graph $G=(V_n,E) $ is stable:
\begin{itemize}
\item[(a)]
only if each of the (strongly) connected components of $G$ is stable;
\item[(b)]
only if for every $k \in \{1,2,\ldots,n\}$ there exists a $k$-decomposition in $G$;
\item[(c)]
if $G$ has a sequence of nested $k$-decompositions $D_1\subset D_2 \subset \ldots \subset D_n$, $k=1,2,\ldots,n$. \end{itemize}
\end{Theorem}

The first two conditions are necessary and the last condition is sufficient. We observe that a graph with one node and a loop is stable.
Obtaining necessary and sufficient conditions in the general case of digraphs is still an open problem.

A zero-pattern is called {\bf symmetric} if the locations of the free variables $\ast$ are symmetric with respect to the main diagonal. More precisely, the zero-pattern $Z_\alpha$ is symmetric if $(i,j) \in E \leftrightarrow (j,i) \in E$.  Note that matrices in a symmetric zero-pattern are not necessarily symmetric. 
We note that in the case of symmetric zero patterns, for every edge $(i, j)$ in its corresponding graph $G$, the opposite edge $(j, i)$ is also in $G$. Therefore, we can consider $G$ to be an undirected graph, possibly containing loops. The notion of $k$-decomposition is naturally carried over to undirected graphs. Quite remarkably, we can show that in the symmetric case, the necessary condition $2$ of Theorem~\ref{th:stabcond} implies the sufficient condition $3$, yielding the following statement:
\begin{Theorem}[\cite{kirkoryan2014decentralized}, Theorem~6]\label{th:resdet}
Let $G$ be a graph corresponding to a symmetric sparse matrix space. Then $G$ is stable \emph{if and only if}:
\begin{enumerate}
\item[(a)] Every node in $G$ is connected to a self-loop.
\item[(b)] The graph $G$ contains an $n$-decomposition.
\end{enumerate} 
\end{Theorem}

\section{Random graphs and zero patterns}

We now introduce the random graphs and random zero-patterns  models considered in this paper.  We can see from Theorem~\ref{th:resdet} that loops in $G$ (resp.  diagonal entries in the corresponding zero pattern $Z$) play a special role in structural stability. Therefore, we treat them differently from edges between distinct nodes (resp. off-diagonal entries in $Z$). We will consider the following two random graph models, which are slight extensions of the ones introduced in the seminal paper~\cite{erdos59}: 

\begin{Definition}[Model A]
Let $p,q \in [0,1]$. A random  (undirected) graph $\mathcal{G}_{p, q}^{n}$ is a random variable which takes values in the set of graphs $\Omega_n$ on vertices $\{1,\ldots,n\}$, such that for every $u, v \in V, u > v,$ the edge $(u, v)$ belongs to $E$ with probability $p$, and for every $u \in V$, the loop $(u,u)$ belongs to $E$ with probability $q$. 
\end{Definition}

To generate a random graph from model A, we throw a biased coin with probability of tails $p$ for each potential edge, and place an edge if the outcome of the throw is tails; no edge otherwise. We then perform the same procedure for all possible loops, but with probability $q$. For any set $S \subseteq \Omega_n$, we have that $\mathbb{P}(G_n \in S)= \mathbb{P} (S)$ where $G_n$ is sampled from model A.

\begin{Definition}[Model B]
A random  graph $\mathcal{G}^{n}_{N, M} = (V, E)$   is a random variable which takes values in the set of  graphs on vertices $\{1,\ldots,n\}$ with {\it exactly} $M$ loops and $N$ non-loop edges, such that each element in this set has equal probability. 
\end{Definition}
To generate a random graph from model $B$, it suffices to place $M$ loops uniformly at random in the graph, and $N$ non-loop edges uniformly at random as well.

We call {\bf graph property} 
a function from the set of graphs to $\{0,1\}$. Being strongly connected or structurally stable are therefore graph properties. We use bold capital letters to denote graph properties. For a graph property $\bH$, we interchangeably say $\bH(G)=1$ or {\it $G$ has property $H$}. We have the following definition:

\begin{Definition}[Monotone properties]
A property $\bH$ of a directed or undirected graph $G$ is called monotone, if adding  {\it edges} to the graph preserves the property.
\end{Definition}

Connectivity and existence of perfect matching are examples of monotone properties of graphs. Being a tree is an example of a non-monotone property.
We require below the following result about monotone properties. It is a straightforward generalization of~\cite{bolob}, Theorem~2.1, to encompass graphs  with loops. 

\begin{Theorem}\label{ivanpenev}
If $H$ is a monotone graph property  and $0 \leq p_1 \leq  p_2 \leq 1$, $0 \leq q_1 \leq q_2 \leq 1$, then
$$\mathbb{P}(\bH(\mathcal{G}^n_{p_1,q_1})=1) \leq \mathbb{P}(\bH(\mathcal{G}^n_{p_2,q_2})=1).$$
\end{Theorem}

The graph properties relevant to structural stability are the following: 
\begin{Definition}
Let $G$ be a random graph with loops. By $\bS, \bH, \bL$ we denote the following graph properties:
\begin{description}
    \item[$\bS:$ ] \hspace{.1cm} $\bS(G) =1$ if $G$ is structurally stable, and zero otherwise;
    \item[$\bH:$ ] $\bH(G) =1$ if $G$ contains a Hamiltonian decomposition, and zero otherwise;
    \item[$\bL:$ ] \hspace{.1cm}$\bL(G)=1 $ if each component of $G$ contains a self-loop, and zero otherwise.
\end{description}
\end{Definition}
These properties are clearly monotone. From Theorem~\ref{th:resdet}, we obtain that a symmetric graph $G$ is stable if and only if $\bH(G)\bL(G)=1$, or, more succinctly, $\bS(G)=\bH(G)\bL(G)$.

\begin{Remark}
From Theorem~\ref{theoremtable1}, it is clear that a graph with a self-loop and a Hamiltonian cycle is stable. It was shown in~\cite{} that for an Erd\"os-Renyi model,

Probability of having a Hamiltonian cycle is low, we need to look at decompositions. threshold for Hamiltonian cycle is $p = \frac{\log n +\log \log n + \omega(1)}{n},$ where $\omega$ is any function tending to infinity when $n \to \infty$. By contrast, our threshold will be much smaller

\end{Remark}

\section{Stability of Random Graphs}

We now derive the probability that a random graph sampled from model A or model B is stable. 
The proof of the main results for both models relies on first establishing a characterization of symmetric digraphs {\it without}  Hamiltonian decompositions.  We call such digraphs {\bf thin}, since all matrices in the corresponding zero patterns  have a determinant equal to zero.

\begin{Definition}[Thin symmetric digraphs]
A symmetric digraph $G$ on $n$ nodes is called {\bf thin} if it does not have a Hamiltonian decomposition. We denote by $\cT^n$ the set of all thin graphs on $n$ nodes.
\end{Definition}

The characterization of thin graphs provided in the next two Lemmas is a fundamental ingredient to the proofs below, and follows closely~\cite{erdos59}:

\begin{Lemma}\label{stupid}
A symmetric digraph $G$ is thin \emph{if and only} if it contains an independent set $I = \{u_1, u_2, \ldots , u_k\}$, such that $|N(I)| = k - 1$ for some $1\leq k\leq n$.
\end{Lemma}
The proof  is a straightforward application of Hall's marriage theorem. 
\begin{proof}
To the digraph $G=(V, E)$, we can assign uniquely the bipartite graph $B_G = (V', V'', E^*)$, where $V' = \{1', 2', \ldots , n'\}$, $V'' = \{1'', 2'', \ldots , n''\}$, and $(i, j) \in E$ if and only if $(i', j') \in E^*$. It is easy to see that $G$ contains a Hamiltonian decomposition {\it if and only if} $B_G$ contains a perfect matching.

First, assume that there exists an independent set $I = \{u_1, u_2, \ldots , u_k\} \subset G$ such that $|N(I)| < k$. Then the same is true for the corresponding set $I' = \{u_1', u_2', \ldots ,u_k'\} \subset V_1$ in the bipartite graph. Therefore, applying Hall's Theorem~\cite{hallmt}\footnote{Recall that Hall's Theorem, when applied to a bipartite graph $B=(V_1 \cup V_2,E)$ with equal size node sets (i.e. $|V_1|=|V_2|$), says that there exists a perfect matching {\it if and only if} for all subsets $W \in V$, such that $W \in V_1$ or $W \in V_2$, we have $|W| \leq |N(W)|$.}, we conclude that $B_G$ does not contain a perfect matching, and therefore $G$ does not contain a Hamiltonian decomposition.

Now assume that $G$ does not contain a Hamiltonian decomposition, and thus $B_G$ does not contain a perfect matching. Applying Hall's Theorem again, we conclude that there exists a subset $I' = \{u_1', u_2', \ldots , u_k'\} \subset V_1$, such that $N(I') < k$. We decompose $I'$ into two disjoint subsets as follows:  $I' = I_1' \cup I_2'$, where $I_1' = \{u_i' \in I'\mid u_i'' \in N(I')\}, I_2'=\{u_i' \in I' \mid u_i'' \notin N(I')\}$. By construction, $|N(I')| \geq |I_1'|$. We conclude from this observation, and the fact that $|I_1'|+|I_2'|=k$, that the set $I_2'$ is non-empty, because otherwise $|N(I')| \geq k$. We now show that the set $I_2 \subseteq V$ corresponding to $I_2'$ is independent with $N(I_2) <|I_2|$. To this end, set $I_1'' = \{u_i'' \in V''|u_i' \in I_1'\}$; if the node $u_i'' \in I_1''$, then from the definition of $I_1'$,  $u_i'' \in N(I')$. Furthermore, if a node $u_i'' \in N(I_2')$, then $u_i'' \notin N(I')$ by definition of $I_2'$. Thus $N(I_2') \cap I_1'' = \emptyset$ and  
$$|N(I_2')| \leq |N(I')| - |I_1''| < k - |I_1'| = |I_2'|.$$

Therefore, the  set $I_2 \in V$ is indeed an independent set with $|N(I_2)| < |I_2|$. Now choose the smallest independent set $I$, such that $|N(I)| < |I|$. If $|N(I)| < |I| -1 $, then we can remove any vertex $v$ from $I$ and get an independent set $J = I \setminus v$ for which $|N(J)| < |J|$. This is a contradiction, and thus $|N(I)| = |I| - 1$, which concludes the second part of the proof.
\end{proof}

We now exhibit a decomposition of the set $\cT^n$ of thin digraphs into disjoint subsets.

\begin{Definition}\label{def:defK}
We let $F_k \subseteq \Omega_n$  be the following event: a graph $G=(V,E)$ on $n$ nodes belongs to $F_k$ if and only if:
\begin{enumerate}\item 
there exists  an independent set $I \subseteq V$ with $k$ vertices, such that $|N(I)| = k - 1$;
\item for every independent set $J\subseteq V$ with $l < k$ vertices, $|N(J)| \neq l - 1$.
\end{enumerate}
\end{Definition}
Note that with a slight abuse of notation, we use interchangeably symmetric digraphs an undirected graphs.  We have the following Lemma:

\begin{Lemma}\label{lem:TunionFk} The sets $F_k$ from  Def.~\ref{def:defK} are \emph{disjoint} and  $\cup_{k=1}^{\lceil(n+1)/2\rceil} F_k= \cT^n$.
\end{Lemma}
\begin{proof}
	
If $G \in F_k$, then by condition 2, it does not have an independent set of cardinality $l-1$, for any $l<k$, with $|N(J)|=l-1$. Hence, such $G$ cannot belong to any $F_l$, $l < k$ and we conclude that $F_l \cap F_k = \emptyset$ for $l<k$, from which the first assertion follows. 

Furthermore, if $G$ contains an independent set $I$ of size $m$ with $|N(I)|=m-1$, then $m + m - 1 \leq n$, and therefore $G$ belongs to some $F_k$, $k \leq \lceil (n+1)/2 \rceil$.  Therefore, $\cT^n = \cup_{k=1}^{\lceil(n+1)/2\rceil} F_k$.
\end{proof}

\subsection{Structural stability for model A}

We now address random graphs sampled from model A specifically. The first step is to characterize the general structure of graphs  in $\G^n_{p,q}$ for $p=\frac{\log(n) + c + o(1)}{n}$. In particular,  the number of isolated vertices follows a Poisson distribution, and most graphs are made of isolated vertices and a large component (of size more than half the number of nodes).

\begin{Proposition}\label{prop:bollo}
Let $p= \frac{\log(n) + c + o(1)}{n}$ with $c \in \R$. Then, the number of isolated vertices in the graph $G = \G^n_{p,q}$ converges in distribution to a Poisson random variable $P_{\lambda}$, where $\lambda = e^{-c}$. 
\end{Proposition}
A proof of this well-known statement can be found in, e.g.,~\cite[Theorem 3.1]{frieze2015book}.
We will also need the following result:
\begin{Lemma}\label{lem:tech}
Let $p= \frac{\log(n) + c + o(1)}{n}$ with $c \in \R$. The probability that $\G^n_{p,q}$  has a connected component of size strictly larger than $1$ and smaller or equal to $n/2$ is $o(1)$.
\end{Lemma}

\begin{proof}
Denote by $\bP(k,n,p)$ the probability that $G =\G^n_{p,q}$ has a connected component of size $k$. Recalling that if a subset of cardinality $k$ has a connected component then  it contains a spanning tree, and that the number of spanning trees is $k^{k-2}$, we obtain the bound 
$$\bP(k,n,p) \leq {n \choose k} k^{k-2} p^{k-1} (1-p)^{k(n-k)}.$$

Replacing $p$ by its value in the Lemma's statement, and using Stirling's approximation for the binomial coefficient, we get 

\begin{align*}\bP(k,n,p) &\leq \frac{e^kn^k}{k^k} k^{k-2}  \left(\frac{\log(n) + c + o(1)}{n}\right)^{k-1} \left(1-\frac{\log(n) + c + o(1)}{n}\right)^{kn/2}\\
&\ll    \frac{e^kn^k}{k^k} k^{k-2}  \left(\frac{\log(n)}{n}\right)^{k-1}  e^{-(\log(n)-c)k/2}\\
&\ll    e^{k(1+c/2)} k^{-2}  n \left(\log(n)\right)^{k-1}  \frac{1}{n^{k/2}}.\\
\end{align*} For $k \in \{5,\ldots,n/2\}$, each $\bP(k,n,p)$ is $o(1/n)$, and for $k=3,4$, $\bP(k,n,p) = o(1).$ For $k=2$, we have directly

$$\bP(2,n,p) \ll Cn^2 \frac{\log n}{n} \left( 1-\frac{\log n}{n} \right)^{2(n-2)} \ll C  n\log n  \frac{1}{n^2} \left( 1-\frac{\log n}{n} \right)^{-4} = o(1).$$
We thus conclude that $\sum_{k=1}^{n/2} \bP(k,n,p) =o(1).$
\end{proof}

We now introduce the following two events in $\Omega_n$: let ${\cal I}^{n,k}$ is the event that $G$ has {\it exactly} $k$ isolated vertices, and ${\cal J}^{n,k}$ the event that $G$ has exactly $k$ isolated vertices and a {\it unique} connected component of size $n-k$. As a Corollary of the previous result, we show that for Erd\"os-R\'enyi graphs, ${\cal I}^{n,k}$ and ${\cal J}^{n,k}$ are very close:

\begin{Corollary}\label{coro:ia}
Let $p= \frac{\log(n) + c + o(1)}{n}$ with $c \in \R$ and $G = \G^n_{p,q}$.  Then, for $0 \leq k \leq n$, $\bP({\cal I}^{n,k} - {\cal J}^{n,k})=o(1)$. 
\end{Corollary}

\begin{proof}
The event ${\cal I}^{n,k}$ is the disjoint union of ${\cal J}^{n,k}$ and the event that $G$ has $k$ isolated vertices and two or more components of size larger than 1. Since at least one of these components must have size smaller or equal to $n/2$, the result now follows from Lemma~\ref{lem:tech}. 
\end{proof}

\begin{restatable}[]{Proposition}{propunioA}
\label{prop:unioA}
For $0 \leq p, q \leq 1$ given rates, the probability that $G=\mathcal{G}^n_{p,q}$  is unstable is asymptotically equal to the probability that $G$ contains a component without a self-loop, i.e.
\begin{align}
    \mathbb{P}(\bS({G}) = 0) = \mathbb{P}(\bL({G}) = 0) + o(1).
\end{align}
\end{restatable}

\begin{proof}
From Theorem \ref{th:resdet} we know that $G$ is stable if and only if it contains a Hamiltonian decomposition and all of its nodes are connected to loops. Since $F_1 \subset \{G \mid \bL({G}) = 0\}$ and 
$$\mathbb{P}(\cT^n) = \mathbb{P}(F_1) + \sum_{k = 2}^{\lceil(n + 1)/2\rceil} \mathbb{P}(F_k),$$
we have:
\begin{align}
\mathbb{P}(\bS({G}) = 0) - \mathbb{P}(\bL({G}) = 0) \leq \sum_{k = 2}^{\lceil(n + 1)/2\rceil} \mathbb{P}(F_k).
\end{align}
Next, we prove that
$$\sum_{k = 2}^{\lceil(n + 1)/2\rceil} \mathbb{P}(F_k) = o(1),$$
which will conclude the proof.

To this end, fix $k > 1$. We denote by $I$ an independent set satisfying the first condition of Def.~\ref{def:defK} for a given graph in $F_k$.   We can choose its $k$ vertices in $\binom{n}{k}$ ways and their $k - 1$ neighbors in $\binom{n - k}{k-1}$ ways. Now, if we assume that a vertex $v \in N(I)$ is adjacent to only {\it one} vertex $u \in I$, then $J = I \setminus u$ will be such that $|J| = k -1$ and $|N(J)| = k -2$; but this is a contradiction with condition 2. Therefore, every vertex $v \in N(I)$ is adjacent to at least {\it two} vertices $u_1, u_2 \in I$.

Putting the above facts together, we  have the following upper bound:
\begin{equation}\label{eq:upboundFk}
\mathbb{P}(F_k) \leq \binom{n}{k} (1 - p)^{\binom{k}{2}} \binom{n - k}{k - 1}(1 - p)^{(n - 2k + 1)k}\left(\binom{k}{2} p^2\right)^{k-1},
\end{equation}
where the term $\binom{n}{k} (1 - p)^{\binom{k}{2}}$ accounts for the choice of $k$ distinct vertices without edges between them, the term $\binom{n - k}{k - 1}(1 - p)^{(n - 2k + 1)k}$ accounts for the choice of their $k-1$ neighbors, and the fact that no nodes outside the $k-1$ selected ones can be connected to $I$. Finally,  the last term accounts for the just-derived condition that $v \in N(I)$ implies that $v$ is the neighbor of at least {\it two} vertices in $I$.

We now bound the right-hand-side of Eq.~\eqref{eq:upboundFk}. Expanding the binomial coefficients and gathering terms of the same power, we get
\begin{equation}\label{eq:p411}  \mathbb{P}(F_k) \leq \frac{n!}{k!(n-k)!}\frac{(n-k)!}{(k-1)!(n-2k+1)!}\frac{k^{k-1}(k-1)^{k-1}}{2^{k-1}}p^{2(k-1)}(1-p)^{(n -1.5k +0.5)k}.\end{equation}

We first consider the terms $F_k$ with $2 \leq k < \lceil \frac{(n+1)}{2}\rceil$ (specifically, if $n$ is odd, the term $F_{\frac{n+1}{2}}$ is omitted, otherwise all terms are considered.) Recall that for all positive integers $k$, Stirling's approximation yields the following inequalities: \begin{equation}\label{eq:stirling}
\sqrt{2\pi}k^{k+\frac{1}{2}}e^{-k} \leq k! \leq e k^{k+\frac{1}{2}}e^{-k}.\end{equation}
Using these in the above expression, we obtain\footnote{By $f(x)\ll g(x)$ it is meant that for $x$ large enough, $f(x) < c g(x)$ for some $c>0$ (Vinogradov symbol).} 
\begin{align*}
    \mathbb{P}(F_k) 
    &\ll \frac{n^{n+0.5}k^{k-1}(k-1)^{k-1}}{k^{k+0.5}(k-1)^{k-0.5}(n-2k + 1)^{n-2k+1.5}2^{k-1}}p^{2(k-1)}(1-p)^{(n-1.5k+0.5)k} \end{align*}
Now, using the following relation $$\frac{k^{k-1}(k-1)^{k-1}}{k^{k+0.5}(k-1)^{k-0.5}} = \frac{1}{k^{1.5}(k-1)^{0.5}} \ll \frac{1}{k^2},$$ we obtain $$
 \mathbb{P}(F_k)\ll \frac{n^{n+0.5}}{k^2(n-2k + 1)^{n-2k+1.5}2^{k-1}}p^{2(k-1)}(1-p)^{(n-1.5k+0.5)k}.
$$
To proceed, we substitute $p = \frac{\log n + c + o(1)}{n}$. We have that $$p^{2(k-1)} = \left(\frac{\log n + c + o(1)}{n}\right)^{2(k-1)} = \frac{(\log n)^{2k-1}}{n^{2k-1}} \frac{(\log n +c + o(1))^{2(k-1)}}{(\log n)^{2(k-1)}}.$$ We let $C>1$ and get:
 \begin{align*}
    \mathbb{P}(F_k)&\ll \frac{n^{n+0.5}}{k^2(n-2k + 1)^{n-2k+1.5}2^{k-1}}\frac{(\log n)^{2(k-1)}}{n^{2(k-1)}}C^{2(k-1)}(1-p)^{(n-1.5k+0.5)k}\\
    &\ll \left(1 + \frac{2k-1}{n-2k + 1}\right)^{n - 2k +1.5}\frac{n(\log n)^{2(k-1)}}{k^22^{k-1}}C^{2(k-1)}(1-p)^{(n-1.5k+0.5)k},
        \end{align*} where we used the fact that $\left(1 + \frac{2k-1}{n-2k + 1}\right)^{n - 2k +1.5}= \left(\frac{n}{n-2k+1}\right)^{n-2k+1.5} $.        
For $2 \leq k < \lceil \frac{(n+1)}{2}\rceil$, a short calculation shows that $\frac{n-2k+1.5}{k} < \frac{10(n-2k+1)}{2k-1}$. Hence

    \begin{align*}\mathbb{P}(F_k)&\ll \left(\left(1 + \frac{2k-1}{n-2k+1}\right)^\frac{10(n-2k+1)}{2k - 1} C^2n^{\frac{1}{k}}(\log n)^2(1-p)^{n - 1.5k + 0.5}   \right)^k, \\
    &\ll \left(C^2n^{\frac{1}{k}}(\log n)^2 (1-p)^{n-1.5k+0.5}\right)^k,
 \end{align*}
\noindent where the last inequality follows from the fact that the function $f(x)=(1+\frac{1}{x})^x$ takes values between $1$ and $e$.

We now expand  $\log(1-p)$ in Taylor series and, recalling that $p= \frac{\log n + c + o(1)}{n}$, we get
\begin{equation*}\label{eq:lntaylor} %ex bumbum
1-p = \exp(\log(1-p)) = \exp\left(-\frac{\log(n)(1+o(1))}{n}\right) = n^{-\frac{1+o(1)}{n}}.
\end{equation*}

Therefore,
\begin{align}\label{eq:PFkforallk}
   \mathbb{P}(F_k) \ll \left(C^2(\log n)^2 n^{-\frac{n-1.5k+0.5}{n}(1 +o(1))+\frac{1}{k}}\right)^k \ll \left(C^2\frac{(\log n)^2}{n^{0.1}}\right)^k,
\end{align}
where the last inequality stems from the fact that $\frac{n-1.5k+0.5}{n}(1 +o(1))+\frac{1}{k}>0.1$ for $2 \leq k < \lceil \frac{n+1}{2} \rceil$.

We now bound the size of $F_k$ for the case $n$  odd and $k = (n + 1)/2$, which is not included in the above analysis (note that we divide  there by $n-2k+1$). The analysis follows the same line as the one above, and we thus provide fewer details. For $n$ is  sufficiently large, we have from Eq.~\eqref{eq:p411}

\begin{align}\label{penev}
    \mathbb{P}(F_k) &\leq \frac{n!}{\left(\frac{n+1}{2}\right)!\left(\frac{n-1}{2}\right)!} (1-p)^{\binom{(n+1)/2}{2}} \binom{\frac{n+1}{2}}{2}^{\frac{n-1}{2}} p^{n-1}  \notag \\
     &\ll \frac{n^{n+\frac{1}{2}}}{(n^2-1)^{\frac{1}{2}}(n+1) 2^{\frac{n-5}{2}}} n^{-\frac{n^2-1}{8n}(1+o(1)) } \left(\frac{C^{\frac{1}{2}}\log n}{n}\right)^{n-1} \notag \\
    &\ll \left(n^{2-\frac{1}{n+1}} n^{-\frac{1}{4}\left(1-\frac{1}{n}\right)(1+o(1))}\left(\frac{C^{\frac{1}{2}}\log n}{n}\right)^{2-\frac{4}{n+1}}\right)^{\frac{n+1}{2}}   \notag  \\
    &\ll \left(C\frac{(\log n)^2}{n^{0.25 + o(1)}}\right)^k \ll \left(C\frac{(\log n)^2}{n^{0.1}}\right)^k,   
\end{align}
where $C = \left(1+\frac{1+c}{\log 2}\right)^2$.

Combining \eqref{eq:PFkforallk} and \eqref{penev}, we see that $\sum_{k = 2}^{[(n+1)/2]}\mathbb{P}(F_k) = o(1)$, which concludes the proof.
\end{proof}

Recall that having a component without a self-loop implies instability. Thus, the previous result says that the transition to instability takes place shortly before the transition to having graphs with multiple components, not all of them having a self-loop. We now evaluate this latter probability:

\begin{Proposition}\label{prop:mid}
Let $p= \frac{\log(n) + c + o(1)}{n}$ and $\lambda = e^{-c}$. The probability that all components of the graph $G=\mathcal{G}^n_{p,q}$ contain a node with a self-loop is equal to $$\mathbb{P}(\bL(G) = 0) = \left\lbrace \begin{aligned} 
e^{-\lambda} -e^{-\mu-\lambda} + o(1) & \mbox{ for } q = \frac{\mu+o(1)}{n}, \quad 0 \geq \mu \\
e^{-\lambda(1-\mu)} + o(1) & \mbox{ for } q= \mu +o(1), \quad 0< \mu < 1
\end{aligned} \right.  $$
\end{Proposition}

\begin{proof} 
Recall that ${\cal J}^{n,k}$ is the event that the graph $G$ has $k$ isolated vertices and a component of size $n-k$.  Conditioned on  ${\cal J}^{n,k}$,  the probability that at least one component has no self-loop is equal to 
\begin{align}
\mathbb{P}(\bL(G = 0) \mid  {\cal J}^{n, k}) =  1 - q^k + q^k (1-q)^{n-k}.
\end{align}

We have \begin{equation}\mathbb{P}((\bL(G) = 0) \cap  {\cal J}^{n,k})= \mathbb{P}((\bL(G) = 0) \mid   {\cal J}^{n,k}) \mathbb{P}( {\cal J}^{n,k})= (1 - q^k + q^k (1-q)^{n-k}) \mathbb{P}( {\cal J}^{n,k}).\end{equation}

Now, recall that $\mathcal{I}^{n,k}$  is the event that $G$ has $k$ isolated vertices. From Prop.~\ref{prop:bollo}, we know that $\bP(\mathcal{I}^{n-k})$ converges to a Poisson distribution with parameter $\lambda$.  From Cor.~\ref{coro:ia}, we have that $\bP({\cal I}^{n,k} - {\cal J}^{n,k})=o(1)$, and thus

\begin{align}
    \mathbb{P}(\bL(G) = 0) &=  \sum_{k=0}^{n}\mathbb{P}(\mathcal{I}^{n, k}) - \sum_{k=0}^{n}\mathbb{P}(\mathcal{I}^{n, k}) q^k + (1-q)^n\sum_{k=0}^{n}\mathbb{P}(\mathcal{I}^{n, k})\left(\frac{q}{1-q}\right)^k +o(1) \\
    &= 1 - \sum_{k=0}^{n}\frac{(\lambda q)^k e^{-\lambda}}{k!} + (1-q)^n\sum_{k=0}^{n} \frac{\left(\frac{\lambda q}{1-q}\right)^k e^{-\lambda}}{k!} + o(1) \\
    &= 1 - e^{-\lambda (1-q)}\mathcal{P}_{\lambda q}(X \leq n) + (1-q)^n e^{\frac{\lambda q}{1-q}-\lambda} \mathcal{P}_{\frac{\lambda q}{1-q}}(X \leq n) + o(1),
\end{align}
where $\mathcal{P}_{\lambda}(X \leq n)$ denotes the cumulative distribution function of a Poisson distribution with parameter $\lambda$.

We now analyze two  asymptotic regimes for $q=q(n)$. Note that when the parameter $\lambda(n)$ of a Poisson distribution  is upper-bounded, then its cumulative distribution $ \mathcal{P}_{\lambda(n)}(X \leq n)$ converges to one as $n \to \infty.$
\begin{description}
\item[(a) $q = \mu + o(1) = \frac{\omega(1)}{n}$:] in this case, $q$ converges to a constant $0< \mu< 1$ slower than $\beta/n$, for any $\beta >0$. We have that  $(1 - q)^n = o(1)$, $\frac{\lambda q}{1-q} = O(1)$, and
\begin{align*}
 \mathbb{P}(\bL(G) = 1) = e^{-\lambda(1-\alpha)} + o(1).
 \end{align*}

\item[(b) $q = \frac{\mu + o(1)}{n}$:] in this case, $q$ converges to 0 as $\mu/n$, for $\mu \geq 0$. Using the Taylor expansion of $\log(1-q)$, we have
 \begin{align*}
 \mathbb{P}(\bL(G) = 1) &= e^{-\lambda} - e^{n\log (1-q) + \frac{\lambda q}{1-q} - \lambda} \mathcal{P}_{\frac{\lambda q}{1-q}}(X \leq n) + o(1) \\
 &= e^{-\lambda} - e^{-n(\frac{\mu}{n}+ o(\frac{1}{n})) + \frac{\lambda(\mu + o(1))}{n - \mu + o(1)} - \lambda}\mathcal{P}_{o(1)}(X \leq n) \\
 &= e^{-\lambda} + e^{-\mu -\lambda + o(1)}(1 - o(1)) \\
 &= e^{-\lambda} - e^{-\mu - \lambda} + o(1).
 \end{align*}
 \end{description}
\end{proof}

We now summarize our results and characterize the structural stability of random graphs as a function of $p$ and $q$:

\begin{Theorem}\label{theoremtable1}
 Consider the random graphs $G=\mathcal{G}^n_{p,q}$ and let $\lambda = e^{-c}$ for a parameter $c$. Then, the probability that $G$ is structurally stable is given by:
 \renewcommand{\arraystretch}{1.5}
 \begin{center}
 \begin{tabular}{||c| c c c||} 
 \hline
 \diagbox[width=7em]{$q(n)$}{$p(n)$} &$ \frac{\log(n) - \omega(1)}{n}$ &$\frac{\log(n) + c + o(1)}{n}$ &  $\frac{\log(n) + \omega(1)}{n}$ \\ [0.5ex] 
 \hline\hline
 $\frac{\mu + o(1)}{n} \geq 0$ & $o(1)$ & $e^{-\lambda}(1 - e^{- \mu}) + o(1)$ & $1 - e^{-\mu} + o(1)$\\
 \hline
  $0< \mu + o(1)=\omega(1/n)$ & $o(1)$ for $\mu < 1$ & $e^{-\lambda(1-\mu)} + o(1)$ & $1 - o(1)$ \\
 \hline
\end{tabular}
\end{center} 

\end{Theorem}

\begin{proof}
The entries of the column with $p =\frac{\log(n) + c + o(1)}{n} $ were derived in Prop.~\ref{prop:mid}. The entries of the left and right columns are obtained using simple monotonicity argument. For the left column, note that for all $c$,  $\frac{\log(n) - \omega(1)}{n} \ll \frac{\log(n) + c + o(1)}{n}$. Hence, for $q(n)$ fixed,

$$\bP\left(G \mbox{ is stable} \mid p=\frac{\log(n) - \omega(1)}{n}\right) \leq  \bP\left(G \mbox{ is stable} \mid p=\frac{\log(n) + c + o(1)}{n}\right)$$ 
holds asymptotically for all $c$. Taking $c$ arbitrarily large and negative, $\lambda$ becomes arbitrarily large and positive, and the previous inequality establishes the left column. For the right column, we have similarly that, asymptotically, $$\bP\left(G \mbox{ is stable} \mid p=\frac{\log(n) + \omega(1)}{n}\right) \geq  \bP\left(G \mbox{ is stable} \mid p=\frac{\log(n) + c + o(1)}{n}\right).$$ Taking $c$ arbitrarily large positive yields $\lambda$ converging to zero and we obtain the right column. 
\end{proof}

\begin{figure}
\centering
\includegraphics{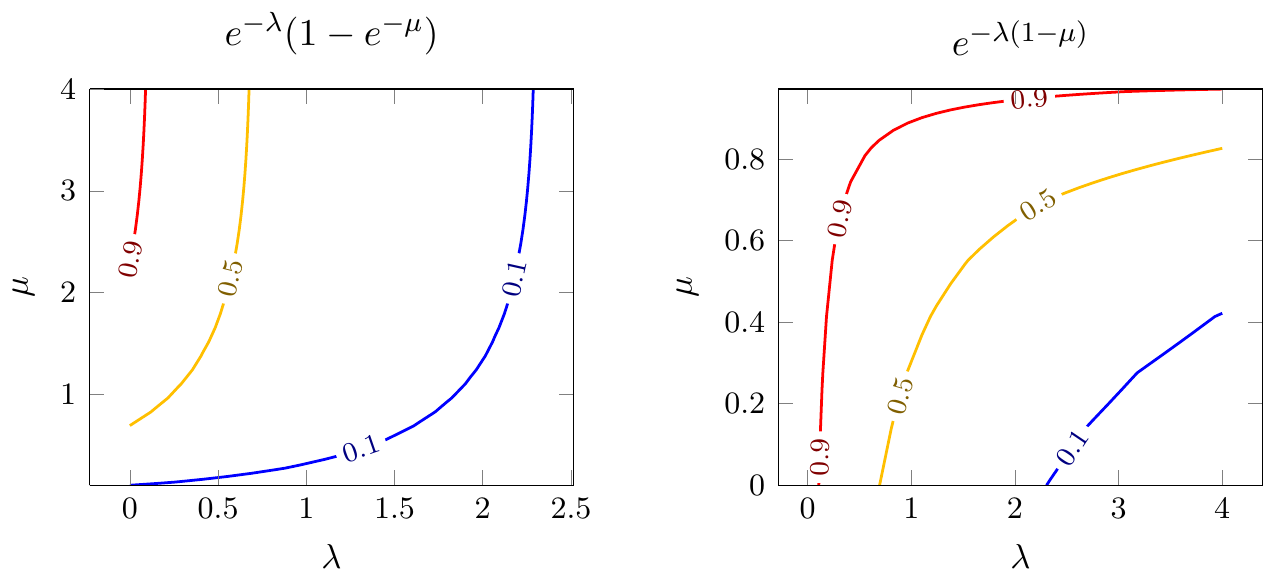}
%\begin{tikzpicture}
%
% \begin{axis}[
%    at={(0cm,0cm)},
%    title={$e^{-\lambda}(1-e^{-\mu})$},
%    enlarge x limits,
%    view={0}{90},
%    xlabel=$\lambda$, ylabel=$\mu$,
%    small,
%]
%\addplot3[domain=0:4,
%        domain y=0:4,
%        contour gnuplot={levels={0.1,0.5,0.9}},
%        thick,samples=30,samples y=30,
%    ] {exp(-x)*(1-exp(-y))};
%\end{axis}
%
%\begin{axis}[
%    at={(7cm,0cm)},
%    title={$e^{-\lambda(1-\mu)}$},
%    enlarge x limits,
%    view={0}{90},
%    xlabel=$\lambda$, ylabel=$\mu$,
%    small,
%]
%\addplot3[domain=0:4,
%        domain y=0:4,
%        contour gnuplot={levels={0.1,0.5,0.9}},
%        thick,samples=30,samples y=30,
%    ] {exp(-x*(1-y))};
%\end{axis}
%
%
%
%\end{tikzpicture}

\caption{The curves represent isolevel sets of the probability that a graph $G={\cal G}^n_{p,q}$ is stable in terms of the parameter $\lambda$ and $\mu$, when $p(n) =\frac{\log(n) + c + o(1)}{n}$ and $q(n)=\mu+o(1)=\omega(1/n)$ (left), $q(n)=o(1)$ (right). } 
\end{figure}

\subsection{Model B}
We now derive the probability that a random graph sampled from model B is stable, where we recall that model B has two parameters: $N$, the number of edges between distinct nodes, and $M$, the number of loops.

The first step, as in the derivation for model A, is to establish the number of independent components in a random graph sampled from model B. To this end, we have the following result (see Prop.~\ref{prop:bollo}):
\begin{Proposition}[Erdos-Renyi]\label{prop:numisomodelB}
Let $G=\G^n_{N,0}$ with $N=\frac{n}{2} (\log(n)+c+o(1))$, $c \in \R$. Then, the number of isolated vertices in $G$ converges  in distribution to a Poisson distribution with parameter $\lambda=e^{-c}$.
\end{Proposition}

Proposition \ref{prop:numisomodelB} and its proof can be found in \cite[Theorem 3]{connectrenyi}. The following Lemma is the equivalent of Lemma~\ref{lem:tech}:
\begin{restatable}[]{Lemma}{lemtechtwo}:
\label{lemma:tech2}
Let $N= \frac{n}{2}(\log(n) + c +o(1))$ with $c \in \R$. The probability that $G = \G^n_{N,M}$ has a connected component of size strictly larger than $1$  and smaller or equal to $n/2$ is $o(1)$.
\end{restatable}
The proof follows the same general approach as the proof of Lemma~\ref{lem:tech}, but with rather different specifics.

\begin{proof}[Proof of Lemma~\ref{lemma:tech2}] Denote by $\bP(k,N,M)$ the probability that $G \in \G^n_{N,M}$ has a connected component of size $k$. Recall that a subset of cardinality $k$ has a connected component only if  it contains a spanning tree, and that the number of such spanning trees is $k^{k-2}$.  Using this fact, we can establish the following bound on the number of assignments of $N$ edges in a graph with $n$ nodes that yield a connected component of size $k$:

\begin{align*}\bP(k,N,M) &\leq \frac{1}{{n(n-1)/2 \choose N}}{n \choose k} k^{k-2}  {L  \choose N-k+1 }\\
& \leq \frac{N!(\bar{n}-N)!}{\bar{n}!} \frac{n!k^{k-2}}{k!(n-k)!}\frac{L!}{(N-k+1)!(L-N+k-1)!},
\end{align*} where $$L={k \choose 2} -k+1+ {n-k \choose 2}=k(k-1)+\bar n-nk+1,$$ and   $\bar{n} = \binom{n}{2} = n(n-1)/2$. To see this, we count the number of assignments of $N$ edges in a graph with $n$ nodes that 
result in at least one connected component of size $k$. Indeed, the first term is the number of all possible assignments of $N$ edges in a graph with $n$ nodes, the second term is the number of subsets of cardinality $k$, the third term is the number of spanning trees in such a subset, and the last term counts the possible position of the remaining edges.

Note that we can rewrite $L$ as \begin{equation}\label{eq:defL}L-N+k-1 = \bar n - N +k(k-1)-nk+1-N+k-1= \bar n - N-k(n-k).\end{equation}

We recall that $N = n(\log n + c + o(1))/2$ and establish some inequalities that hold for $2 \leq k \leq \lceil n/2 \rceil$.
\begin{enumerate}
    \item \begin{align*}\frac{N!}{(N-k+1)!} &\leq N^{k-1} =(n(\log n + c + o(1))/2)^{k-1}\\ &=\exp((k - 1) (\log n)(1 + o(1))),\end{align*}
    where we used the fact that $n\log n = n^{1+o(1)}.$
    
    \item $\frac{n!k^{k-2}}{k!(n-k)!} \leq \frac{n^k k^{k-2} e^k}{k^{k+\frac{1}{2}}} \leq (ne)^k \leq \exp(k (\log n)(1+o(1))).$

    \item 
    \begin{align*}\frac{(\bar n -N)!}{(L-N+k-1)!}\frac{L!}{\bar n!} &= \frac{(\bar n - N)\cdots(\bar n - N - k(n-k) + 1)}{\bar n\cdots(\bar n -k(n-k) -k  + 2)}.
    \end{align*}
    After pairing the terms in the numerator with the first $k-1$ terms in the denominator, we get
    \begin{small}
    \begin{align*}
    \frac{(\bar n -N)!}{(L-N+k-1)!}\frac{L!}{\bar n!}  &\leq  \frac{(\bar n - N)\cdots(\bar n - N - k(n-k) + 1)}{\bar n \cdots(\bar n - k(n-k) + 1)(\bar n - k(n-k)) \cdots(\bar n -k(n-k) -k  + 2)}  \\
    &= \prod_{i = 0}^{k(n-k)-1} \left(\frac{\bar n - N - i}{\bar n - i}\right)\prod_{i = 0}^{k - 2}\frac{1}{\bar n - k(n - k) - i}.
    \end{align*}
    \end{small}
    Note that in the first product, the term corresponding to $i=0$ is the largest, and in the second product, the term corresponding to $i = k-2$ is the largest. Therefore,
    \begin{small}
    \begin{align*}
        \frac{(\bar n -N)!}{(L-N+k-1)!}\frac{L!}{\bar n!}
 &\leq \left(\frac{\bar n - N}{\bar n}\right)^{k(n-k)} \left(\frac{1}{\bar n - k(n-k) -k + 2}\right)^{k-1} \\
    &= \exp\left(-k(n-k)\frac{\log n}{n}(1+o(1)) - (k-1)(\log n)(2 + o(1))\right) \\
    &\leq \exp\left((2 - 3k + \frac{k^2}{n})(\log n)(1 + o(1))\right).
    \end{align*}
    \end{small}

\end{enumerate}
Putting the above inequalities together, we obtain

\begin{equation}\label{lemdd:int1}
    \bP(k,N,M) \leq \exp\left(-\left(k - 1 - \frac{k^2}{n}\right)(\log n)(1 + o(1))\right).
\end{equation}

When $k = 2$, we see that $\mathbb{P}(k, N, M) = o(1)$. When $k \geq 3$, we see that $\mathbb{P}(k, N, M) = o(n^{-1.5})$. Therefore,
$$
\sum_{k = 2}^{\lceil\frac{n}{2}\rceil} \mathbb{P}(k, N, M) = o(1).
$$
\end{proof}

The Proposition below says that the transitions to stability and  to having a graph with a component without a self-loop take place at almost the same time asymptotically. To show it, we follow a general approach first introduced in~\cite{erdos59}:

\begin{Proposition}
For $N, M$ given rates, the probability that $G=\mathcal{G}^n_{N, M}$ is unstable is asymptotically equal to the probability that $G$ contains a component without a self-loop, i.e.
\begin{align}
    \mathbb{P}(\bS(G) = 0) = \mathbb{P}(\bL(G) = 0) + o(1).
\end{align}
\end{Proposition}
\begin{proof}

Recall that $\cT^n$, the set of thin graphs, can be decomposed as the disjoint union of the sets $F_k$ defined in Def.~\ref{def:defK}. We will bound the measure of the set
\begin{align*}
    \bP\left(\cup_{k=2}^{\lceil(n + 1)/2\rceil} F_k\right) = \sum_{k = 2}^{\lceil(n + 1)/2\rceil} \bP(F_k)
\end{align*}
and show that it is $o(1)$.

Let $k \geq 2$ and consider a graph $G=(V,E)$ in $F_k$, i.e. $G$ has an independent set $I$ of cardinality $k$, with $|N(I)|=k-1$ and is so that there is no independent set $J$ in $G$ with $|J|=l$, $|N(J)|=l-1$. We can choose the $k$ vertices of the independent set $I \subset V$ in $n \choose k$ ways. Then, we can choose their $k-1$ neighbors in $n - k \choose k - 1$ ways. We can show, similarly to what is done in the proof of Prop.~\ref{prop:unioA},  that every vertex $v \in N(I)$ must be adjacent to at least  {\it two} vertices $u_1, u_2 \in I$. Hence,  $2(k-1)$ edges are needed to satisfy the above conditions on $I$ and its neighbor set. For the placement of the remaining $N - 2(k - 1)$ edges there are $\binom{n}{2} - \binom{k}{2} - k(n - 2k + 1) - 2(k - 1)$ possibilities. 

Therefore, setting
$$L=\binom{n}{2} - \binom{k}{2} - k(n - 2k + 1) - 2(k - 1),$$
we get the upper bound:

$$
\mathbb{P}(F_k) \leq \binom{n}{k}\binom{n - k}{k - 1}\binom{k}{2}^{k - 1} \binom{L}{N - 2k + 2}/{\binom{\binom{n}{2}}{N}},
$$
where the term $\binom{n}{k}$ accounts for the choice of $k$ distinct vertices without edges between them, the term $\binom{n - k}{k - 1}$ account for the choice of their $k - 1$ neighbors, the term $\binom{k}{2}^{k - 1}$ accounts for the condition that $v \in N(I)$ implies that $v$ is the neighbor of at least two vertices in $I$, the term $\binom{L}{N - 2k + 2}$ accounts for the choice of the remaining $L$ edges in the graph, and the term $\binom{\binom{n}{2}}{N}$ accounts for all possible choices of the $N$ edges.

First, we consider the terms $F_k$ with $2 \leq k < \lceil \frac{(n + 1)}{2} \rceil$. Expanding the binomial coefficients and using Stirling's approximation \eqref{eq:stirling}, we get
\begin{align*}
    \mathbb{P}(F_k) &\ll \frac{n^{n+0.5}k^{k-1}(k-1)^{k-1}}{k^{k+0.5}(k-1)^{k-0.5}(n-2k + 1)^{n-2k+1.5}2^{k-1}}\left(\frac{N}{\binom{n}{2}}\right)^{2(k-1)}\left(\frac{L}{\binom{n}{2}}\right)^{N-2k + 2}.
\end{align*}

Next, we combine terms of the same power and obtain
\begin{small}
\begin{align*}
\mathbb{P}(F_k) &\ll \left(\left(1 + \frac{2k - 1}{n - 2k + 1}\right)^{\frac{n - 2k + 1.5}{k}} \frac{n^{2 - \frac{1}{k}}}{2} \left(\frac{\log(n) + c + o(1)}{n - 1}\right)^{\frac{2(k-1)}{k}} (1 - Q)^{\frac{N - 2k + 2}{k}}\right)^k,
\end{align*}
\end{small}

where
\begin{align*}
0 \leq Q = 1 - \frac{L}{\binom{n}{2}} = \frac{\binom{k}{2} + kn - 2k^2 + 3k - 2}{\binom{n}{2}} = \frac{-3k^2 + 5k + 2kn - 4}{n^2 - n} < 1.
\end{align*}

We note that $\frac{n - 2k + 1.5}{k} < \frac{10(n - 2k + 1)}{2k - 1}$ and recall that the function $(1 + \frac{1}{x})^x$ is bounded between $0$ and $e$ for $x > 0$. Let $C > 0$ be such that 
$$\left(1 + \frac{2k - 1}{n - 2k + 1}\right)^{\frac{n - 2k + 1.5}{k}}\left(\frac{\log n + c + o(1)}{\log n}\right)^2 < 2C.$$
Therefore,
\begin{align*}
\mathbb{P}(F_k)
    &\ll \left(C n^{\frac{1}{k}}(\log n)^2  (1 - Q)^{\frac{N}{k}}\right)^k.
\end{align*}

Expanding $\log (1 - Q)$ in Taylor series, we get

\begin{align*}
(1 - Q)^{\frac{N}{k}}  &= \exp\left(\log\left(1 - Q\right)\frac{n(\log(n) + c + o(1))}{2k}\right) \\
&= \exp\left(\frac{(3k^2 - 5k - 2kn + 4)\log(n)(1 + o(1))}{2kn}\right) \\
&= \exp\left(-\left(1 - \frac{3k}{2n}\right)\log(n)(1 + o(1))\right) \\
&= n^{-\left(1 - \frac{3k}{2n}\right)(1 + o(1))}.
\end{align*}

Therefore, 

\begin{align}\label{mufk1}
\mathbb{P}(F_k) \ll \left(\frac{C(\log n)^2}{n^{(1 - \frac{3k}{2n} - \frac{1}{k})(1 + o(1))}}\right)^k \ll \left(C\frac{(\log n)^2}{n^{0.01}}\right)^k,
\end{align}
where the last inequality follows from the fact that $$
\frac{3k^2}{2(0.99k - 1)} \leq n
$$
for every $2 \leq k \leq \frac{n + 1}{2}$.

In a similar fashion, we can show that for $n$ odd number, and $k = \frac{n - 1}{2}$, we have
\begin{align}\label{mufk2}
    \mathbb{P}(F_k) \ll \left(C\frac{(\log n)^2}{n^{0.01}}\right)^k.
\end{align}
Combining \eqref{mufk1} and \eqref{mufk2}, we have that 
$ \sum_{k = 2}^{\lceil(n + 1)/2\rceil} \bP(F_k) = o(1)$, which concludes the proof.

\end{proof}

\begin{Proposition}\label{prop:mid2}
Let $N= \frac{n}{2}(\log(n) + c)$, $c, \mu \in \R$,  $\lambda = e^{-c}$ and consider a graph $G = \G^n_{N,M}$. The probability that all components of $G$ contain a node with a self-loop is given by
\begin{equation}\mathbb{P}(\bL(G) = 1)= \begin{cases} e^{-\lambda} +o(1)&\mbox{ for } M(n) = \mu + o(1), 0 < \mu. \\ e^{-\lambda(1-\mu)} + o(1) & \mbox{ for } M(n) = \mu n +o(n), 0 < \mu \leq 1

  \end{cases}\end{equation}

\end{Proposition}

\begin{proof} 
The proof follows the same general approach as the one of Prop.~\ref{prop:mid}, though the specifics are rather different.  As before, we let ${\cal J}^{n,k}$ be the event that the graph $G$ has $k$ isolated vertices and a component of size $n-k$, and ${\cal I}^{n,k}$ the event that the graph has $k$ isolated vertices.  Conditioned on  ${\cal J}^{n,k}$,  the probability that at least one component has no self-loop is equal to 

\begin{align}
\mathbb{P}((\bL(\mathcal{G}^n) = 0) \mid  {\cal J}^{n, k}) = \left\lbrace \begin{aligned} 1-\frac{{n-k \choose M-k}}{{n \choose M}}= 1- \frac{(n-k)!M!}{n!(M-k)!}& \mbox{ when } k < M \\
1 & \mbox{ when } k \geq M
\end{aligned}\right.
\end{align}

From Lemma~\ref{lemma:tech2}, we have that  $\bP({\cal I}^{n,k} - {\cal J}^{n,k})=o(1)$, from Prop.~\ref{prop:numisomodelB},  we know that $\bP(\mathcal{I}^{n-k})$ converges to a Poisson distribution with parameter $\lambda=e^{-c}$. Therefore, we obtain $$\mathbb{P}((\bL(\mathcal{G}^n) = 1)) = \sum_{k=0}^\infty \mathbb{P}((\bL(\mathcal{G}^n) = 1) \mid  {\cal J}^{n, k}) \mathbb{P}({\cal I}^{n,k})+ o(1).$$

We first assume that  $M = \mu n +o(n)$. We have  
\begin{align*}\frac{M!}{(M-k)!} &= M(M-1)\cdots(M-k+1)\\
&= M^k(1-\frac{1}{M}) \cdots (1-\frac{k}{M})\\
&= M^k (1-\frac{s_2}{M}+\cdots+\frac{s_{k}}{M^k}),
\end{align*}
where $s_i=\begin{bmatrix} k \\ k-i\end{bmatrix}$ is  the unsigned Stirling number of the first kind,  and it is easy to see that $s_i \leq k!$, $1 \leq i \leq k$. Similarly, we have

$$\frac{(n-k)!}{n!} =  n^{-k}(1-\frac{s_1}{n}+\cdots\pm \frac{s_{k-1}}{n^k})^{-1}.$$

We now evaluate the probability as follows: first, we split the sum according to

\begin{align*}
        \sum_{k=0}^{M-1} \frac{\lambda^k e^{-\lambda}}{k!} \frac{(n-k)!M!}{n!(M-k)!} &=      \sum_{k=0}^{\log \log M} \frac{\lambda^k e^{-\lambda}}{k!} \frac{(n-k)!M!}{n!(M-k)!} + \sum_{k=\log \log M+1 }^{M-1} \frac{\lambda^k e^{-\lambda}}{k!} \frac{(n-k)!M!}{n!(M-k)!} 
\end{align*}

For the first term,
\begin{align*}
        \sum_{k=0}^{\log \log M} \frac{\lambda^k e^{-\lambda}}{k!} \frac{(n-k)!M!}{n!(M-k)!}&= \sum_{k=0}^{\log \log M} \left(\frac{\lambda M}{n}\right)^k \frac{e^{-\lambda}}{k!} \frac{(1-\frac{s_1}{M}+\cdots)}{(1-\frac{s_1}{n}+\cdots)}\\   &=e^{-\lambda+\lambda \mu} {\cal P}_{\lambda \mu+o(1)}(X <\log  \log M)+o(1),
\end{align*}
where the last line comes from the fact that $(\log \log x)!/x = o(1/\log x)$ and $s_i \leq (\log \log M)!$, and the fact that ${\cal P}_{\lambda \mu+o(1)}(X <\log  \log M)=1+o(1)$.

For the second term, 

\begin{align*}
        \sum_{k=\log \log M+1 }^{M-1} \frac{\lambda^k e^{-\lambda}}{k!} \frac{(n-k)!M!}{n!(M-k)!} &\leq       \sum_{k=\log \log M+1}^{n} \frac{\lambda^k e^{-\lambda}}{k!}\\
        &\leq {\cal P}_{\lambda}(X > \log \log M) = o(1).
\end{align*}

We now consider the case where $M(n)=\mu +o(1)$, $\mu >0$, i.e. the number of edges converges to a finite, non-zero constant. In this case, $(n-k)!/n! = (n(n-1)\cdots(n-k))^{-1} = n^{-k}+o(1)$ and we obtain 

\begin{align}
\mathbb{P}((\bL(\mathcal{G}^n) = 1) &= \sum_{k=0}^{M-1}    \lambda^k \frac{e^{-\lambda}}{k!}\frac{(n-k)!M!}{n!(M-k)!} \\
& = \sum_{k=0}^{M-1}    \left(\frac{\lambda}{n}\right)^k \frac{e^{-\lambda}}{k!}\frac{M!}{(M-k)!} + o (1).
\end{align}

Since $M!/(M-k)! \geq 1$ for $ k = 0,\ldots, M-1$, we have
\begin{align}
    \mathbb{P}((\bL(\mathcal{G}^n) = 1) & \geq \sum_{k=0}^{M-1}    \left(\frac{\lambda}{n}\right)^k \frac{e^{-\lambda}}{k!} + o (1)\\
    &\geq e^{\lambda/n-\lambda} \sum_{k=0}^{M-1}    \left(\frac{\lambda}{n}\right)^k \frac{e^{-\lambda/n}}{k!} + o (1)\\ & 
    \geq e^{-\lambda}  {\cal P}_{\lambda/n}(X < M) + o (1).
\end{align}

We also have $ \frac{M!}{(M-k)!}   \leq M^k$. Therefore 
\begin{align}
    \mathbb{P}((\bL(\mathcal{G}^n) = 1) & \leq \sum_{k=0}^{M-1}    \left(\frac{\lambda M}{n}\right)^k \frac{e^{-\lambda}}{k!} + o (1)\\
    &\leq e^{\lambda M/n-\lambda} \sum_{k=0}^{M-1}    \left(\frac{\lambda M}{n}\right)^k \frac{e^{-\lambda M/n}}{k!} + o (1)\\ & 
    \leq e^{-\lambda}  {\cal P}_{\lambda M /n}(X < M) + o (1).
\end{align}
As in the first case, ${\cal P}_{\lambda M /n}(X < M) \to 1$ and the upper and lower bounds agree to $e^{-\lambda}.$
\end{proof}

We can summarize the above results in the following Theorem, whose proof is similar to the one of Theorem~\ref{theoremtable1}, but relies on Prop.~\ref{prop:mid2} instead of Prop.~\ref{prop:mid}.
\begin{Theorem}\label{th:table2}
 Consider the random graphs $G=\mathcal{G}^n_{N, M}$ and let $\lambda = e^{-c}$ for a parameter $c \in \R$. Then, the probability that $G$ is structurally stable is given by
 \renewcommand{\arraystretch}{1.5}
 \begin{center}
 \begin{tabular}{||c| c c c||} 
 \hline
 \diagbox[width=7em]{$M(n)$}{$N(n)$} &$ \frac{n(\log(n) - \omega(1))}{2}$ &$\frac{n(\log(n) + c + o(1))}{2}$ &  $\frac{n(\log(n) + \omega(1))}{2}$ \\ [0.5ex] 
 \hline \hline
  $ \mu  + o(1)$, $0<\mu $ & $o(1)$ & $e^{-\lambda} + o(1)$ & $1 - o(1)$ \\ 
  \hline
 $ \mu n + o(n)$ for $0< \mu < 1$ & $o(1)$ (\mbox{for} $0<\mu < 1$)  & $e^{-\lambda(1-\mu)} + o(1)$ & $1 - o(1)$ \\
  \hline
\end{tabular}
\end{center} where we recall that $\omega(1)$ is any positive unbounded function of $n$. 
\end{Theorem}

\paragraph{Hamiltonian cycles and stability} It was shown in~\cite{komlos1983limit}, based on ideas from~\cite{posa1976hamiltonian}, that for a random graph on $n$ nodes with 
$$N=\frac{n}{2}(\log n+\log\log n + c)$$ edges, the (asymptotic) probability of it containing a Hamiltonian {\it cycle} obeys
$ P= e^{-\lambda},$ for $\lambda =e^{-c}.$  From Theorem~\ref{theoremtable1}, it is clear that a graph with {\it one} self-loop and a Hamiltonian {\it cycle} is stable, hence

$$ P(\G_{N,1}^n \mbox{ stable} )\geq e^{-\lambda},$$ for $N$ as above.
We see from Theorem~\ref{th:table2} that, when constraining the graph to have a unique self loop, allowing Hamiltonian decompositions (instead of only allowing cycles)  improves the rate by a term of $\log \log n$.

\section{Summary and outlook}

A zero-pattern is a vector space of matrices with entries either fixed to zero or  arbitrary real.  We can represent these vector spaces as $0/\ast$ matrices, where the $\ast$ means that the entry is arbitrary real. We have investigated in this paper the structural stability of  zero patterns where the positions of the $\ast$ are symmetric about the main diagonal, but otherwise randomly placed. As is usually done, we represent zero-patterns by their corresponding digraphs (which are here symmetric digraphs), and called these digraphs stable if the corresponding zero-patterns contained a stable (Hurwitz) matrix.  We evaluated the probability that a random symmetric digraph sampled from an Erd\"os-R\'enyi model is stable. In particular, we considered 
\begin{itemize}\item a model in which one places a non-diagonal $\ast$-entry at a given position in the zero-pattern according to a Bernoulli random variable with parameter $p$, and a diagonal $\ast$-entry according to a Bernoulli with parameter $q$, with all random variables pairwise independent (model A),
\item  a model in which one places exactly $N$ non-diagonal $\ast$-entries  and $M$ diagonal $\ast$-entries in the zero pattern (model B). 

\end{itemize}
For each model, we derived the asymptotic (in the number of nodes) probability that a zero pattern sampled from them is stable, and the results are summarized in Theorems~\ref{theoremtable1} and~\ref{th:table2} for models A and B respectively.

In the sparse regime (left column), both models behave similarly and stable graphs are rare. Indeed, with few edges, the graphs have many connected components, and the probability of each of them having a self-loop was found to be vanishingly small. 

In the transitional (middle column) and dense graphs (right column) regimes, the results are affected by the number of self-loops in each model. In model $A$, the ratio of nodes with self-loop over the total number of nodes is a random variable, with expected value $\mu$, whereas in model $B$, this ratio is {\it deterministic}, with value $\mu$ as well. For both models, when the expected ratio is small (first row of the tables), we find that the probability that the graph is stable is approximately equal to the probability that it is connected and that it contains  {\it at least} one self-loop. 

This result is fairly intuitive in view of the necessary and sufficient conditions for a graph to be stable given in Theorem~\ref{th:stabcond}. The difference in the two models in these regimes is explained by the fact that in model $B$, $\mu>0$ guarantees that the graph contains {\it at least one} self-loop, whereas in model $A$, it does not (since the ratio is a random variable).

When the ratio grows larger, one can show that its  variance in the case of  model $A$ converges to zero, and its expected value is the same as the one of model B. Hence both models behave similarly in that regime.

\bibliographystyle{plain}
\bibliography{references.bib}
\end{document}